%% file: main.tex
\documentclass[11pt]{article}
\usepackage{graphicx}
\usepackage{amsthm, amsmath, amssymb, tikz, bm}
\usepackage{bbm}
\usepackage{mathtools, intcalc}
\usepackage{xifthen}
\usepackage[colorlinks=true,
linkcolor=blue,citecolor=blue,
urlcolor=blue]{hyperref}

\usepackage[margin=1in]{geometry} 

\usepackage[shortlabels]{enumitem}
\usepackage{todonotes}
\usetikzlibrary{calc,shapes, backgrounds}
\allowdisplaybreaks

\newtheorem{lemma}{Lemma}
\newtheorem{theorem}{Theorem}

\newcommand{\ds}{\displaystyle}

\newcommand{\lp}{\left(}
\newcommand{\rp}{\right)}

\newcommand{\lcm}{{\rm lcm}}
\usepackage[normalem]{ulem}

\pgfdeclarelayer{edgelayer}
\pgfdeclarelayer{nodelayer}
\pgfsetlayers{edgelayer,nodelayer,main}

\title{On the size of outerplanar graphs with positive Lin-Lu-Yau Ricci curvature}

\author{
Xiaonan Liu
\thanks{Vanderbilt University, Nashville, TN, 37240, ({\tt xiaonan.liu@vanderbilt.edu}).}\and
Linyuan Lu
\thanks{University of South Carolina, Columbia, SC 29208,
({\tt lu@math.sc.edu}). This author was supported in part by NSF grant DMS-2038080.} \and
Zhiyu Wang \thanks{Louisiana State University, Baton Rouge, LA, 70810 (zhiyuw@lsu.edu). This author was supported in part by LA Board of Regents grant LEQSF(2024-27)-RD-A-16.} 
}

\begin{document}
\maketitle
\begin{abstract}
In this paper, extending a result of Brooks et.al.~[\textit{arXiv:2403.04110}], 
we show that if an outerplanar graph $G$ with minimum degree at least $2$ has positive Lin-Lu-Yau curvature on every vertex pair, then $G$ has at most $10$ vertices, and this upper bound is sharp. 
\end{abstract}

\section{Introduction}\label{sec:ricci}
Ricci curvature plays an important role in the geometric analysis of Riemannian manifolds. There are many variants of Ricci curvature in discrete spaces e.g., the Bakry and Emery's ``lower Ricci curvature bound"~\cite{Bakry-Emery1985} and Ollivier's coarse Ricci curvature in metric spaces \cite{Ollivier} (also see \cite{Sammer}), etc. The first definition of Ricci curvature on graphs was introduced by Chung and Yau in~\cite{Chung-Yau96}. To obtain a good log-Sobolev inequality, they defined \textit{Ricci-flatness} in graphs. Later, Lin and Yau \cite{LY} gave a generalization of the lower Ricci curvature bound in the framework of graphs. In \cite{LLY}, Lin, Lu, and Yau modified Ollivier's Ricci curvature \cite{Ollivier} and defined a new variant of Ricci curvature on graphs, which does not depend on the idleness of the random walk. See e.g., \cite{Cho-Paeng2013, CKKLP2021, Paeng2012} and references within for more connections between these variants of Ricci curvature and other graph properties.

Infinite planar graphs are often treated as the discrete version of noncompact simply connected $2$-dimensional manifolds. Thus, it is natural to consider curvature of planar graphs. One of the most well studied notions of curvature in planar graphs is the \textit{combinatorial curvature}. 
Let $G$ be a simple connected planar graph that is $2$-cell embedded in the sphere, and let $V(G)$, $E(G)$, $F(G)$ be the set of vertices, edges, and faces of $G$ respectively. Given $v \in V(G)$, let $N_G(v)$ denote the neighborhood of $v$, i.e., $N_G(v) = \{u\in V(G): vu\in E(G)\}$ and let $d_G(v)= |N_G(v)|$. Moreover, define the closed neighborhood of $v$, denoted by $N[v]$, as $N[v]:=N(v)\cup \{v\}$.
The \emph{combinatorial curvature} of a vertex $v\in V(G)$ in $G$, denoted by $\phi_G(v)$, is defined by $\phi_G(v) = 1 - \frac{d_G(v)}{2} + \sum_{\sigma \in F_G(v)} \frac{1}{|\sigma|}$,
where $F_G(v)$ denotes the set of faces containing $v$, and $|\sigma|$ is the size of $\sigma$, i.e., the number of edges bounding $\sigma$. We often drop the subscript $G$ if it is clear from the context.

A graph $G$ is said to have positive combinatorial curvature, or (combinatorially) positively curved, if $\phi(v) > 0$ for all $v\in V(G)$.
In 2001, Higuchi \cite{Higuchi01} conjectured that if $G$ is a simple, connected, (combinatorially) positively curved graph embedded into a $2$-sphere with minimum degree $\delta(G)\geq 3$, then $G$ is finite. Higuchi's conjecture was verified by Sun and Yu \cite{Sun-Yu04} for cubic planar graphs and later confirmed by DeVos and Mohar \cite{DeVos-Mohar07} as follows.

\begin{theorem}[\cite{DeVos-Mohar07}]\label{thm:DM}
Suppose $G$ is a connected simple graph embedded into a $2$-dimensional topological manifold $\Omega$ without boundary and $G$ has minimum degree at least $3$. If $G$ has positive combinatorial curvature, then it is finite and $\Omega$ is homeomorphic to either a $2$-sphere or a projective plane. Moreover, if $G$ is not a prism, an antiprism, or one of their projective plane analogues, then $|V(G)| \leq 3444$.
\end{theorem}

The minimum possible constants for $|V(G)|$ in Theorem \ref{thm:DM} for $G$ embedded in a $2$-sphere and projective plane respectively was studied in \cite{RBK05, Nicholson-Sneddon11, Chen-Chen08, Zhang08, Oh17}. Recently, answering a question of DeVos and Mohar~\cite{DeVos-Mohar07}, Ghidelli \cite{Ghidelli17} improved the upper bound in Theorem \ref{thm:DM} to $208$, which is tight by examples constructed in \cite{Nicholson-Sneddon11, Oldridge17}.

In \cite{Lu-Wang2023+}, Lu and Wang initiated the study on the size of planar graphs with positive \textit{Lin--Lu--Yau} curvature, \textit{LLY curvature} for short, which is defined on the vertex pairs of a graph (see Section \ref{sec:LLY-curvature} for the definitions). A graph $G$ is called \textit{positively LLY-curved} if it has positive Lin--Lu--Yau curvature on every vertex pair of $G$. In~\cite{Lu-Wang2023+}, Lu and Wang established an analogue of DeVos and Mohar's result in the context of Lin--Lu--Yau curvature as follows.

\begin{theorem}\cite{Lu-Wang2023+}
Let $G$ be a simple positively LLY-curved planar graph $G$ with $\delta(G) \geq 3$. Then the maximum degree $\Delta(G) \leq 17$, and $|V(G)|\leq 17^{544}.$
\end{theorem}

Very recently, Brooks, Osaye, Schenfisch, Wang and Yu~\cite{Brooks2024+} studied the maximum degree and maximum order of positively LLY-curved (maximally) outerplanar graphs. In particular, they showed the following.

\begin{theorem}\cite{Brooks2024+} \label{thm:maxdeg}
    Let $ G $ be a simple positively LLY-curved outerplanar graph with $ \delta(G) \geq 2 $. Then $ \Delta(G) \leq 9$ and the upper bound is sharp.
\end{theorem}

For maximally outerplanar graphs $G$, they are able to show that $|V(G)|\leq 10$ and classify all the positively LLY-curved maximal outerplanar graphs.
\begin{theorem}\cite{Brooks2024+}\label{thm:max_outerplanar}
    Let $G$ be a simple positively LLY-curved maximal outerplanar graph. Then $|V(G)| \leq 10$ and the upper bound is sharp.
\end{theorem}

Is this paper, we extend Theorem \ref{thm:max_outerplanar} to the class of all outerplanar graphs (not necessarily edge-maximal) with minimum degree $2$ and classify all simple positively LLY-curved outerplanar graphs with minimum degree $2$. We first establish the result when $G$ is $2$-connected.

\begin{theorem}\label{thm:outerplanr_2_connected_max_order}
    Let $G$ be a simple $2$-connected positively LLY-curved outerplanar graph. Then $|V(G)|\leq 10$ and the upper bound is sharp.
\end{theorem}

Using Theorem \ref{thm:outerplanr_2_connected_max_order}, we then extend the result to all positively LLY-curved outerplanar graphs with minimum degree at least $2$.

\begin{theorem}\label{thm:outerplanar_max_order}
  Let $G$ be a simple  positively LLY-curved outerplanar graph with $\delta(G)\ge 2$. Then $|V(G)|\leq 10$ and the upper bound is sharp.
\end{theorem}
Note that the condition $\delta(G)\ge 2$ in Theorem~\ref{thm:outerplanar_max_order} is necessary as any star $K_{1,t}$ is positively LLY-curved.

\medskip

{\noindent\bf Organization and Notation.} In Section \ref{sec:LLY-curvature}, we give the definition of Lin-Lu-Yau Ricci curvature.
In Section \ref{sec:prelims}, we will show some lemmas and properties of $2$-connected positively LLY-curved outerplanar graphs, which will be used in the proof of Theorem \ref{thm:outerplanr_2_connected_max_order} and Theorem \ref{thm:outerplanar_max_order}. We prove Theorem \ref{thm:outerplanr_2_connected_max_order} and Theorem \ref{thm:outerplanar_max_order} in Section \ref{sec:max_order_outerplanar}. In the remaining of the paper, we often simply say a graph $G$ is positively curved if it is positively LLY-curved. We use $\kappa_{\textrm{LLY}}(x,y)$ to denote the LLY curvature of the vertex pair $\{x,y\}$. We say an edge $xy\in E(G)$  has positive LLY curvature if $\kappa_{\textrm{LLY}}(x,y)>0$. When there is no confusion, we may just use $\kappa(x,y)$ to denote $\kappa_{\textrm{LLY}}(x,y)$. Moreover, when the vertex pair $\{x,y\}$ is present in two different graphs $G$ and $G'$, we use the notation $\kappa_G(x,y)$ and $\kappa_{G'}(x,y)$ to refer to the LLY curvature of $\{x,y\}$ in $G$ and $G'$ respectively.  Given an outerplanar graph $G$, we say that an edge $e$ is an \textit{exterior} edge if $e$ lies on the boundary of the outer face of $G$, and \textit{interior} edge otherwise.

\section{Lin--Lu--Yau Ricci curvature}\label{sec:LLY-curvature}
 
We adopt the same notation as \cite{LLY}.
A probability distribution (over the vertex set $V(G)$) is a mapping $m: V(G)\to [0,1]$ satisfying $\sum_{x\in V(G)} m(x) = 1$.
Suppose two probability distributions $m_1$ and $m_2$ have finite support. A {\em coupling} between $m_1$ and $m_2$ is a mapping $A: V(G)\times V(G) \to [0,1]$ with finite support so that 
$\sum_{y \in V(G)} A(x,y) = m_1(x) \textrm{ for every $x\in V(G)$ and } \sum_{x\in V(G)} A(x,y) = m_2(y)$ for every $y\in V(G)$. The \textit{Wasserstein-1 transportation distance} between two probability distributions $m_1$ and $m_2$ is defined as follows:
$$W(m_1, m_2) = \inf_A \sum_{x,y\in V(G)} A(x,y) d(x,y),$$
where the infimum is taken over all couplings $A$ between $m_1$ and $m_2$, and $d(x,y)$ is the graph distance between $x$ and $y$. A random walk $m$ on $G=(V,E)$ is defined as a family of probability measures $\{m_v(\cdot)\}_{v\in V(G)}$ such that $m_v(u) = 0$ for all $uv \notin E(G)$. It follows that  $m_v(u) \geq 0$ for all $v,u\in V(G)$ and $\sum_{u\in N[v]} m_v(u) = 1$ for all $v\in V(G)$. 

In~\cite{Ollivier}, Ollivier defined the \textit{coarse Ricci curvature} $\kappa: \binom{V(G)}{2} \to \mathbb{R}$ of a vertex pair in $G$ equipped with a random walk $m$ as follows:
$$\kappa_{\textrm{Ollivier}}(x,y) = 1 - \frac{W(m_x, m_y)}{d(x,y)}.$$
One natural choice of random walks in graphs is the \textit{$\alpha$-lazy random walk} ($0\leq \alpha < 1$) defined as 
$$
m_x^{\alpha}(v) = \begin{cases} 
                        \alpha & \textrm{ if $v=x$,}\\
                        (1-\alpha)/d(x) &\textrm{ if $v\in N(x)$,}\\
                        0 & \textrm{ otherwise.}
                    \end{cases}
$$
In \cite{LLY}, Lin, Lu and Yau defined the Ricci curvature of a vertex pair $\{x,y\}$ in $G$, denoted by $\kappa_{\textrm{LLY}}(x,y)$, based on the $\alpha$-lazy random walk in $G$ as $\alpha$ goes to $1$, i.e., $$\kappa_{\textrm{LLY}}(x,y) = \ds\lim_{\alpha \to 1} \frac{\kappa_{\alpha}(x,y)}{(1-\alpha)},$$
where $\kappa_{\alpha}(x,y) = 1 - \frac{W(m_x^{\alpha}, m_y^{\alpha})}{d(x,y)}$. 
A graph $G$ is called \textit{positively LLY-curved} or positively curved, for short, if $\kappa_{LLY}(x,y) >0$ for all $x \neq y \in V(G)$. 

In \cite{MW}, M\"{u}nch and Wojciechowski gave a limit-free formulation of the Lin-Lu-Yau Ricci curvature using \textit{graph Laplacian}. 

\begin{theorem}\label{thm:curvature_laplacian}\cite{MW} (Curvature via the Laplacian) Let $G= (V,E)$ be a simple graph and let $x \neq y \in V(G)$. Then 

\begin{equation*}
    \kappa_{\textrm{LLY}}(x,y) = \inf_{\substack{f \in Lip(1)\\ \nabla_{yx}f = 1}} \nabla_{xy} \Delta f, 
\end{equation*}
where $\nabla_{xy}f=\frac{f(x)-f(y)}{d(x,y)}$ and $\Delta f(u)=\frac{1}{d(u)}\sum_{v\in N(u)} (f(v)-f(u))$. 
\end{theorem}

Shortly later, Bai, Huang, Lu, and Yau \cite{BHLY} proved a dual theorem for the limit-free definition of the Lin-Lu-Yau Ricci curvature. Given a vertex $x\in V(G)$, Let $m^0_x$ be the probability distribution of random walk at $x$ with idleness equal to $0$, i.e., $m^0_x(y) = \frac{1}{d(x)}$ if $y\in N(x)$, and $m^0_x(y)=0$ otherwise.
For any two vertices $x$ and $y$, a {\em $\ast$-coupling} between $m^0_x$ and $m^0_y$ is a mapping $B: V(G)\times V(G)\to \mathbb{R}$ with finite support such that 
\begin{enumerate}[(1)]
    \item $0<B(x,y)$, but all other values $B(u,v)\leq 0$.
    \item $\sum\limits_{u,v\in V(G)} B(u,v)=0$.  
    \item $\sum\limits_{v \in V(G)} B(u, v)=-m^0_x(u)$ for all $u$ except $x$. 
    \item $\sum\limits_{u \in V(G)} B(u, v)=-m^0_y(v)$ 
for all $v$ except $y$.  
\end{enumerate}

\begin{theorem}\cite{BHLY} (Curvature via the $\ast$-Coupling function)\label{thm:curvatureviacoupling}
For any two vertices $x, y\in V(G)$, 
\begin{equation} \label{eq:curv_coupling}
\kappa_{\textrm{LLY}}(x,y)=\frac{1}{d(x,y)}\sup\limits_{B} \sum\limits_{u, v\in V(G)} B(u, v)d(u, v),
\end{equation}
where the supremum is taken over all $\ast$-couplings $B$ between $m^0_x$ and $m^0_y$.
\end{theorem}

\section{Preliminaries}\label{sec:prelims}


We first show a new way (but follow the spirit of previous formulations) to compute Lin-Lu-Yau curvature lower bounds on a locally finite graph. Given $x\ne y\in V(G)$ with $d(x)\leq d(y)$, let $\lcm(d(x),d(y))$ denote the least common multiple integer of $d(x)$ and $d(y)$. Consider a pair of co-prime integers $c_x$ and $c_y$ satisfying
\[\lcm(d(x),d(y))=d(x)c_x=d(y)c_y.\]
Note that $c_x\ge c_y$. Let $\mu_x$ be a mass distribution defined as
\[
\mu_x(u)=\begin{cases}
c_x-c_y & \mbox{ if } u=y \mbox{ or } u\in N(x)\cap N(y),\\
   c_x &\mbox{ if } u \in N(x)\backslash N[y],\\
0 &\mbox{otherwise.}   
\end{cases}
\]
Let $\mu_y$ be a mass distribution defined as
\[
\mu_y(u)=\begin{cases}
   c_y &\mbox{ if } u \in N(y)\backslash N[x],\\
0 &\mbox{otherwise.}   
\end{cases}
\]
It is straightforward to check that
$\mu_x$ and $\mu_y$ have equal total mass $(d_y-1-|N(x)\cap N(y)|)c_y$
. Given a coupling $\sigma: V(G)\times V(G)\to \mathbb{Z}$ between $\mu_x$ and $\mu_y$ (i.e., $\sigma(u,v)\ge 0$ for any $u,v\in V(G)$, and $\sum_{v\in V(G)} \sigma(z,v)=\mu_x(z)$, $\sum_{u\in V(G)} \sigma(u,z) = \mu_y(z)$ for all vertices $z\in V(G)$), define the {\it transportation cost} of $\sigma$, denoted by $C(\sigma)$, as follows: 
\[C(\sigma)=\sum_{u,v\in V(G)} \sigma(u,v)d(u,v).\] 
Note that by the definition of $\mu_x$ and $\mu_y$, $\mu_x(u)\ge 0$ if $u\in N(x)$ and $\mu_x(u)=0$ otherwise; $\mu_y(u)\ge 0$ if $u\in N(y)\backslash N[x]$ and $\mu_y(u)=0$ otherwise. It follows that for any coupling $\sigma$ between $\mu_x$ and $\mu_y$, $\sigma(u,v)=0$ if $(u,v)\notin N(x)\times (N(y)\backslash N[x])$. Hence, the transportation cost of $\sigma$ is $$C(\sigma)=\sum_{u,v\in V(G)} \sigma(u,v)d(u,v)=\sum_{u\in N(x), v\in N(y)\backslash N[x]}\sigma(u,v)d(u,v).$$

Very recently, Li and Lu~\cite{Li-Lu2024+} showed that for any two vertices $x,y\in V(G)$ in a locally finite graph $G$, 
 \[ \kappa(x,y)= 1+\frac{1}{d(y)} - \frac{\min_{\sigma}C(\sigma)}{\lcm(d(x),d(y))},\]
 where the minimum is taken over all integer-valued couplings $\sigma$ between $\mu_x$ and $\mu_y$.
In this paper, we only need the lower bounds on $\kappa(x,y)$. We reproduce the proof of the lower bounds below for completeness.

\begin{lemma}\label{lem:coupling}
Let $G$ be a locally finite graph, and $x \neq y\in V(G)$ with $d(x)\leq d(y)$.
    For any coupling $\sigma$ between
    $\mu_x$ and $\mu_y$, we have
    \[ \kappa(x,y)\geq 1+\frac{1}{d(y)} - \frac{C(\sigma)}{\lcm(d(x),d(y))}.\]
\end{lemma}
\begin{proof}
We define a $\ast$-coupling $B_\sigma$ between $m_x^0$ and $m_y^0$ as follows.
For any $u,v\in V(G)$, we have
\begin{equation} \label{equ:Bcoupling}
B_\sigma(u,v)=
\begin{cases}
1+\frac{1}{d(y)} & \mbox{ if } (u,v)=(x,y);\\
-\frac{1}{d(y)} & \mbox{ if } u=v \in N[x] \cap N[y];\\
- \frac{\sigma(u,v)}{lcm(d(x),d(y))} & \mbox{ if } (u,v)\in N[x] \times (N(y) \setminus N[x]);\\
0 & \mbox{ otherwise.}
\end{cases}
\end{equation}

Now we verify that $B_\sigma$ is a $\ast$-coupling between $m^0_x$ and $m^0_y$ by checking the four conditions in its definition. Condition (1) is satisfied by the definition.
For Condition (2), we have
\begin{align*}
   \sum\limits_{u,v\in V(G)} B(u,v) &= 1+\frac{1}{d(y)} -\frac{|N[x]\cap N[y]|}{d(y)}
   -\frac{\sum_{u,v}\sigma(u,v)}{lcm(d(x),d(y))}\\
   &= 1+\frac{1}{d(y)} -\frac{|N(x)\cap N(y)|+2}{d(y)}
   -\frac{\|\mu_y\|_1}{lcm(d(x),d(y))}\\
   &=1+\frac{1}{d(y)} -\frac{|N(x)\cap N(y)|+2}{d(y)}
   -\frac{(d_y-1-|N(x)\cap N(y)|)c_y}{lcm(d(x),d(y))}\\
   &=0.
\end{align*}
Both sides of the equation in Condition (3) are zeros unless 
$u\in N[x]$. Since $u\not=x$,  we may assume $u\in N(x)$.
We have
\begin{align*}
    \sum\limits_{v \in V(G)} B(u, v)
    &= -\frac{\mathbbm{1}_{u\in N(x)\cap N[y]}}{d(y)} -
    \frac{\sum\limits_{v \in N(y)\setminus N[x] } \sigma(u,v)}{\lcm(d(x),d(y))}\\
    &= -\frac{c_y \mathbbm{1}_{u\in N(x)\cap N[y]}}{\lcm(d(x),d(y))} - \frac{\mu_x(u)}{\lcm(d(x),d(y))}\\
    &= -\frac{c_x}{\lcm(d(x),d(y))}\\
    &=-\frac{1}{d(x)}\\
    &=-m_x^0(u).
\end{align*}
Now we verify Condition (4). Similarly, we may assume $v\in N(y)$. We divide it into two cases. Case 1: $v\in N(y)\cap N[x]$. We have
\begin{align*}
    \sum\limits_{u \in V(G)} B(u, v)
    &= - \frac{1}{d(y)}\\
    &=-m^0_y(v).
\end{align*}
Case 2: $v\in N(y)\setminus N[x]$. We have
\begin{align*}
    \sum\limits_{u \in V(G)} B(u, v)
    &= -    \frac{\sum\limits_{u \in N[x] } \sigma(u,v)}{\lcm(d(x),d(y))}\\
    &=  - \frac{\mu_y(v)}{\lcm(d(x),d(y))}\\
    &= -\frac{c_y}{\lcm(d(x),d(y))}\\
    &=-\frac{1}{d(y)}\\
    &=-m_y^0(u).
\end{align*}
All four conditions in the definition of $\ast$-coupling are verified.
Therefore,
\begin{align*}
   \kappa(x,y) &\geq \frac{1}{d(x,y)}\sum_{u,v} B_\sigma(u,v)d(u,v) \\
   &= B(x,y) - \sum_{u\in N[x], v\in N(y)\setminus N[x]}\frac{\sigma(u,v)}{\lcm(d(x),d(y))} d(u,v)\\
   &=1+ \frac{1}{d(y)} - \frac{C(\sigma)}{\lcm(d(x),d(y))}. 
\end{align*}\qedhere
\end{proof}

\begin{lemma}\label{lem:pos_degree_pair}
Let $G$ be a graph and $xy$ be an edge in $G$ such that $3=d(x)\le d(y)$. Suppose $N(x)\cap N(y)\ne \emptyset$ (i.e., $x$ and $y$ have one or two common neighbors) and if $d(y)\geq 4$, $y$ has a neighbor $u_2\notin N[x]$ such that $u_2$ is adjacent to some vertex $u_1\in N(x)\cap N(y)$. Then the followings hold. 
\begin{itemize}
\item [(i)] If $|N(x)\cap N(y)|=1$ and $d(y)\le 4$, then $\kappa(x,y)>0$.
\item [(ii)] If  $|N(x)\cap N(y)|=2$ and $d(y)\le 10$, then $\kappa(x,y)>0$.
\end{itemize}
\end{lemma}

\begin{proof}
Suppose $d(y)=k$ for some integer $k\ge 3$. Let $N(y)=\{u_1, u_2, \ldots u_k\}$, where $u_k=x$. We first show (i). Assume that $x$ and $y$ have exactly one common neighbor, say $u_1$. If $d(y)=3$, then $d(x)=d(y)=3$. We have $lcm(d(x),d(y))=3$, and $c_x=c_y=1$. Observe that
$\mu_x(u)=1$ if $u$ is the unique vertex in $N(x)\setminus N[y]$ and zero otherwise; 
$\mu_y(u)=1$ if $u=u_2$ and zero otherwise. It follows that $C(\sigma)\leq 3$ for any coupling $\sigma$ between $\mu_x$ and $\mu_y$. Applying Lemma \ref{lem:coupling}, we then have that 
\[\kappa(x,y)\geq 
1+\frac{1}{d(y)}-\frac{C(\sigma)}{lcm(d(x),d(y))}\geq
1+\frac{1}{3}-\frac{3}{3}=\frac{1}{3}>0.
\]

 We may now assume that $d(y)=4$, and we may assume that $u_2\in N(y)\backslash N[x]$ is adjacent to $u_1$ since $d(y)\ge 4$. Let $x'$ be the neighbor of $x$ other than $u_1, y$. We have
 $lcm(d(x),d(y))=12$, $c_x=4$, and $c_y=3$.
 The mass distribution $\mu_x$
 has $4$ units on $x'$, $1$ unit on both $y$ and $u_1$ while
 $\mu_y$ has $3$ units on both $u_2$ and $u_3$. See Figure \ref{fig:lemma2} (left).

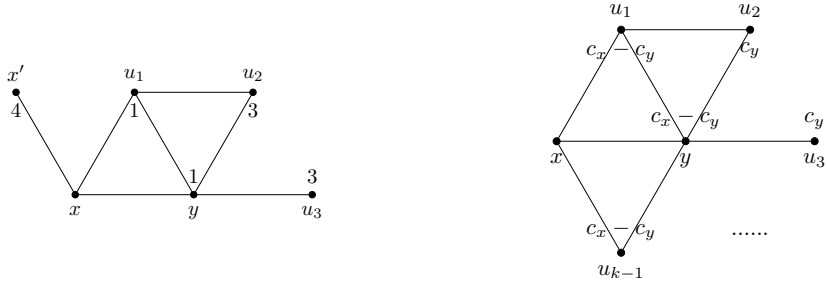
\begin{figure}[htb]
    \begin{center}
        \begin{minipage}{.3\textwidth}
        		\resizebox{4.5cm}{!}{\input{Lemma2_i.tikz}}
        \end{minipage}
            \hspace{2cm}
        \begin{minipage}{.3\textwidth}
        		\resizebox{4cm}{!}{\input{Lemma2_ii.tikz}}
        \end{minipage}
    \end{center}
    \caption{Mass distributions for $\mu_x$ and $\mu_y$.}
    \label{fig:lemma2}
    \end{figure}

  It is straightforward to see that there is a coupling $\sigma$ with the cost $C(\sigma)\leq 1+1+4\times 3=14$.
  Applying Lemma \ref{lem:coupling},
\[\kappa(x,y)\geq 
1+\frac{1}{d(y)}-\frac{C(\sigma)}{lcm(d(x),d(y))}\geq
1+\frac{1}{4}-\frac{14}{12}=\frac{1}{12}>0.
\]
 
Now we show (ii).
We may assume that $N(x)\cap N(y)=\{u_1, u_{k-1}\}$, and $u_2\in N(y)\backslash N[x]$ is adjacent to $u_1$ if $d(y)\ge 4$. See Figure \ref{fig:lemma2} (right) for the mass distribution of $\mu_x$ and $\mu_y$.
Again it is straightforward to see that there is a coupling $\sigma$ of cost
\[C(\sigma)
\leq (1+2+2)(c_x-c_y) - \min\{c_x-c_y, c_y\}.
\]
Applying Lemma \ref{lem:coupling}, we have
\begin{align*}
\kappa(x,y) &\geq 
1+\frac{1}{k}-\frac{  5(c_x-c_y) - \min\{c_x-c_y, c_y\}}{lcm(d(x),d(y))}\\
&=
\begin{cases}
    \frac{5}{k}-\frac{1}{3} & \mbox{ if } k\leq 6;\\
    \frac{7}{k} -\frac{2}{3} & \mbox{ if } 6 \leq k\leq 10; 
\end{cases}\\
&>0.
\end{align*}
Here we apply $k\leq 10$.
\end{proof}

\begin{lemma}\label{lem:exterior_at_least_3}
Let $G$ be an outerplanar graph, and let $uv\in E(G)$ be an exterior edge with positive LLY curvature. Suppose $3\le d(u)\le d(v)$. Then $u,v$ have exactly one common neighbor $w$, and either $d(u)=d(v)=3$ or $d(u)=3, d(v)=4$ and $v$ has a neighbor $y\in N(v)\backslash \{u,w\}$ such that $yw\in E(G)$.
\end{lemma}
\begin{proof}
 Let $uv$ be an exterior edge of $G$. Since $G$ is outerplanar, $u$ and $v$ have at most one common neighbor. We first show that $u$ and $v$ have exactly one common neighbor. Let $C$ be the circuit induced by the outer face of $G$.
 Let $u'$ and $v'$ be the last neighbors of $u$ and $v$ in $uCv$ and $vCu$ respectively (where $uCv$ and $vCu$ are the longer walks between $u$ and $v$ in $C$).
 
 Suppose otherwise that $u$ and $v$ have no common neighbor. Then $u'\ne v'$. Consider the edge $uv$ and the $1$-Lipschitz function $f:N[u]\cup N[v]\to \mathbb{R}$ with 
 $$
     f(x):=\begin{cases}
               -1 & \textrm{ if $x\in N(u)\backslash \{v, u'\}$},\\
                0 & \textrm{ if $x\in \{u,u'\}$},\\
                1 &\textrm{ if $x\in \{v, v'\}$},\\
                2 & \textrm{ if $x\in N(v)\backslash \{u,v'\}$}.
           \end{cases}$$
We then have that
\begin{align*}
    \kappa(u,v) & \leq \Delta f(u)-\Delta f(v) = \frac{1}{d(u)}\lp -(d(u)-2)+1\rp-\frac{1}{d(v)}\lp -1+(d(v)-2)\rp\\
    &\leq \frac{3}{d(u)}+\frac{3}{d(v)}-2 \leq \frac{3}{3}+\frac{3}{3}-2=0,
\end{align*}
since $3\leq d(u)\leq d(v)$, contradicting that the vertex pair $\{u, v\}$ has positive LLY curvature. 

Hence we may assume that $u$ and $v$ have exactly one common neighbor, call it $w$. i.e., $u'=v'=w$. Now consider the edge $uv$ again and the $1$-Lipschitz function $f:N[u]\cup N[v]\to \mathbb{R}$ with 
    $$
     f(x):=\begin{cases}
               -1 & \textrm{ if $x\in N(u)\backslash \{v, w\}$},\\
                0 & \textrm{ if $x\in \{u,w\}$},\\
                1 &\textrm{ if $x=v$ or $x\in (N(v)\cap N(w))\backslash \{u\}$},\\
                2 & \textrm{ if $x\in N(v)\backslash N[w]$}.
           \end{cases}$$
It follows that
\begin{align*}
    \kappa(u,v) & \leq \Delta f(u)-\Delta f(v)\\
    &= \frac{1}{d(u)}\lp -(d(u)-2)+1\rp-\frac{1}{d(v)}\lp -2+(d(v)-2)-\mathbbm{1}_{|N(w)\cap N(v)|=2}\rp\\
    &\leq \frac{3}{d(u)}+\frac{4+\mathbbm{1}_{|N(w)\cap N(v)|=2}}{d(v)}-2.
\end{align*}
Hence if $d(u)=\min\{d(u),d(v)\}\geq 4$, then $\kappa(u,v) \leq \frac{3}{4}+\frac{5}{4}-2\leq 0$, giving a contradiction. Thus we have $d(u)=3$. If $d(v)\ge 5$, then  $\kappa(u,v) \leq \frac{3}{3}+\frac{5}{5}-2\leq 0$, a contradiction. Therefore, we have $d(u)=3$ and $d(v)=3$ or $4$. If $d(u)=3$ and $d(v)=4$, $\kappa(u,v)\leq 0$ unless $|N(w)\cap N(v)|$=2, and it follows that there exists a vertex $y\in N(v)\backslash \{u,w\}$ such that $yw\in E(G)$.
\end{proof}

\begin{lemma}\label{lem:deg2_vtx}
Let $G$ be an outerplanar graph, and $u$ be a degree $2$ vertex in $G$ with $N(u)=\{v,w\}$ and $vw\notin E(G)$. Suppose $uv$ has positive LLY curvature. Then $d(v)\leq 4$. Moreover,
\begin{itemize}
\item[(i)] if $d(v)=4$ then there exists some vertex $y\in (N(v)\cap N(w))\backslash \{u\}$ and another vertex $z\in N(v)\cap N(y)$ (this implies that $u,v,w$ are contained in a $4$-cycle, $uvywu$);
\item [(ii)] if $d(v)=3$ then there exists some vertex $y\in (N(v)\cap N(w))\backslash \{u\}$  (this also implies that $u,v,w$ are contained in a $4$-cycle, $uvywu$);
\item [(iii)] and if $d(v)=2$ then $u,v,w$ are contained in a cycle of length $4$ or $5$.
\end{itemize}
\end{lemma}
\begin{proof}
Let $u$ be a degree 2 vertex in $G$ such that $N(u)=\{v,w\}$, $vw\notin E(G)$ and the edge $uv$ has positive LLY curvature. Consider the edge $uv$ and a $1$-Lipschitz function $f:N[u]\cup N[v]\to \mathbb{R}$ as follows.
$$
     f(x):=\begin{cases}
               -1 & \textrm{ if $x=w$},\\
                0 & \textrm{ if $x=u$ or $x\in (N(v)\cap N(w))\backslash \{u\}$},\\
                1 &\textrm{ if $x=v$ or $x$ is adjacent to some vertex in $(N(v)\cap N(w))\backslash \{u\}$ },\\
                2 & \textrm{ otherwise.}
           \end{cases}
$$
Note that $\Delta f(u)=0$.

If $d(v)\in \{2,3,4\}$, then the consequence in (iii), (ii) or (i) respectively must hold; otherwise, either $(N(v)\cap N(w))\backslash \{u\}=\emptyset$ and by the outerplanarity of $G$ there must exist one neighbor $v'$ of $v$ which is distance $3$ from $w$ (and thus $f(v')=2$), or $(N(v)\cap N(w))\backslash \{u\}\ne\emptyset$ and $d(v)=4$ and the two neighbors $v',v''\notin N(w)$ of $v$ are distance $3$ from $w$ ($f(v')=f(v'')=2$). It follows that
$$\kappa(x,y) \leq -\Delta f(v)\leq -\min \{\frac{1}{d(v)}\lp -1+1\rp, \frac{1}{4} \lp -2+2\rp\}\leq 0,$$
giving a contradiction.

If $d(v)\geq 5$, then there must exist at least two neighbors $v',v''$ of $v$ which are distance $3$ from $w$ (and thus $f(v')=f(v'')=2$). Furthermore, there exists at most one vertex $y$ which is in $(N(v)\cap N(w))\backslash \{u\}$ (and thus $f(y)=0$). Therefore,
$$\kappa_G(x,y) \leq -\Delta f(v)\leq - \frac{1}{d(v)}\lp -1-1+2\rp\leq 0,$$
giving a contradiction.
\end{proof}

\begin{lemma}\label{lem:4-facical-cycle_structure}
Let $G$ be a  $2$-connected positively curved outerplanar graph with outer cycle $C$. Suppose $G$ is not a maximal outerplanar graph and $G$ is not a cycle. Let $D$ be a facial cycle of length at least $4$ of $G$  such that $D\ne C$. Then $D$ has length exactly $4$, and one of the following holds.
\begin{itemize}
\item [(i)] Exactly three consecutive edges of $D$ are exterior edges.
\item [(ii)] Exactly two consecutive edges of $D$ are exterior edges.
\item [(iii)] All edges of $D$ are interior edges and $G$ is an $8$-vertex graph $G_8$ (shown in Figure~\ref{fig:G8}).
\end{itemize}
\end{lemma}
\begin{figure}[htb]
 \centering
	\resizebox{4.5cm}{!}{\input{G8.tikz}}%
	\caption{$G_8$.}
    \label{fig:G8}
    \end{figure}
\begin{proof} Let $D:=v_1v_2\cdots v_k v_1$ for some integer $k\ge 4$. Let $e_i:= v_iv_{i+1}$ for each $i\in [k-1]$ and $e_k:=v_kv_1$. For simplicity, let $v_{k+1}=v_1, e_0=e_k$ and $e_{k+1}=e_1$. 
 Since $G$ is $2$-connected, $e_i$ is contained in exactly two facial cycles, $D$ and another cycle $D_i$. We show that if $e_i=v_iv_{i+1}$ is an interior edge and $d(v_i)\ge 4, d(v_{i+1})\ge 4$ for some $i\in [k]$, then $D_i$ is a triangle. Without loss of generality, we may assume that $e_1=v_1v_2$ is an interior edge and $d(v_1)\ge 4, d(v_2)\ge 4$, and we need to show $D_1$ is a triangle. Suppose that $D_1$ has length at least four. Let $u_j$ be the neighbor of $v_j$ in $D_1$ such that $u_j\ne v_{3-j}$ for $j=1,2$. Note that $u_1\ne u_2$ as $|V(D_1)|\ge 4$. Consider the edge $v_1v_2$ and a $1$-Lipschitz function $f:N[v_1]\cup N[v_2]\to \mathbb{R}$ as follows.
    $$
     f(x):=\begin{cases}
               -1 & \textrm{ if $x\in N(v_1)\backslash \{v_2, u_1, v_k\}$},\\
                0 & \textrm{ if $x\in \{v_1,u_1,v_k\}$},\\
                1 &\textrm{ if $x\in \{v_2, u_2, v_3\}$},\\
                2 & \textrm{ if $x\in N(v_2)\backslash \{v_1, u_2, v_3\}$.}
           \end{cases}$$
 Note that $\nabla_{v_2v_1} f=1, \Delta f(v_1)=\frac{1}{d(v_1)}(-(d(v_1)-3)+1)=\frac{4}{d(v_1)}-1$, and $\Delta f(v_2)=\frac{1}{d(v_2)}(-1+(d(v_2)-3)=-\frac{4}{d(v_2)}+1$. By Theorem~\ref{thm:curvature_laplacian}, $$\kappa(v_1,v_2)\le \Delta f (v_1)-\Delta f (v_2)=\frac{4}{d(v_1)}+\frac{4}{d(v_2)}-2 \le \frac{4}{4}+\frac{4}{4}-2=0,$$ a contradiction. Hence $D_1$ is a triangle. Next we claim that if $e_i=v_iv_{i+1}$ is an exterior edge for some $i\in [k]$, then either $e_{i-1}\in E(C)$ or $e_{i+1}\in E(C)$. Suppose $e_{i-1}, e_{i+1}$ both are interior edges. Then $d(v_i)\ge 3$ and $d(v_{i+1})\ge 3$. Since $v_i, v_{i+1}$ have no common neighbor (as $v_i v_{i+1}$ is in $D$ and is an exterior edge), it follows from Lemma~\ref{lem:exterior_at_least_3} that the edge $e_i=v_iv_{i+1}$ has non-positive LLY curvature. Thus, either $e_{i-1}\in E(C)$ or $e_{i+1}\in E(C)$, and this implies that either $d(v_i)=2$ or $d(v_{i+1})=2$.

First we show that $k=4$. Suppose the length of $D$ is at least five. Since $G$ is not a cycle, there exists some $v_i\in V(D)$ such that $d(v_i)>2$. We may assume $v_1$ has degree at least three. We claim that $d(v_2)\ge 3$. If $v_2$ has degree two, then $v_1,v_2$ are contained in a $4$-cycle by Lemma~\ref{lem:deg2_vtx} as $\kappa(v_1,v_2)>0$, contradicting that $D$ has length at least five. Thus, $v_2$ has degree at least three. Similarly, we have $d(v_i)\ge 3$ for all $i\in [k]$. Then it follows from the claims above that no edge in $D$ is an exterior edge as no vertices in $D$ have degree $2$, i.e., all edges of $D$ are interior edges. Since $G$ is outerplanar, we then obtain that all vertices in $D$ have degree at least four in $G$. Consider the edge $e_1=v_1v_2$, and we may assume that $4\le d(v_1)\le d(v_2)$. Recall that $e_1=v_1v_2$ is contained in exactly two facial cycles, $D$ and $D_1$. We know that $D_1$ is a triangle as $\kappa(v_1, v_2)>0$. Let $D_1=v_1 w v_2 v_1$. Since $G$ is outerplanar, $v_2$ has at most one neighbor in $N(v_2)\backslash \{v_1, v_3, w\}$ that is adjacent to $w$. Consider the edge $v_1v_2$ and a $1$-Lipschitz function 
$f:N[v_1]\cup N[v_2]\to \mathbb{R}$ as follows.
    $$
     f(x):=\begin{cases}
               -1 & \textrm{ if $x\in N(v_1)\backslash \{v_2,w\}$},\\
                0 & \textrm{ if $x\in \{v_1,w\}$},\\
                1 &\textrm{ if $x\in \{v_2, v_3\}$ or $x\in N(v_2)\cap N(w)\backslash \{v_1\}$},\\
                2 & \textrm{ otherwise.}
           \end{cases}$$

It follows that 
\begin{align*}
\kappa(v_1,v_2) & \le  \Delta f(v_1)-\Delta f(v_2)\\
&=\frac{1}{d(v_1)}\lp -(d(v_1)-2)+1 \rp -\frac{1}{d(v_2)}\lp -2+(d(v_2)-3)-\mathbbm{1}_{(N(v_2)\cap N(w))\backslash \{v_1\}\neq\emptyset}\rp\\
&= \frac{3}{d(v_1)} +\frac{5+\mathbbm{1}_{(N(v_2)\cap N(w))\backslash \{v_1\}\neq\emptyset}}{d(v_2)}-2
\end{align*}
Hence if $(N(v_2)\cap N(w))\backslash \{v_1\}=\emptyset$, then 
$$\kappa(v_1,v_2)\leq  \frac{3}{d(v_1)}+\frac{5}{d(v_2)}-2 \le \frac{3}{4}+\frac{5}{4}-2=0,$$
giving a contradiction; otherwise $(N(v_2)\cap N(w))\backslash \{v_1\}\ne \emptyset$ implies that $d(v_2)\ge 5$ and 
$$\kappa(v_1,v_2)\leq  \frac{3}{d(v_1)}+\frac{6}{d(v_2)}-2 \le \frac{3}{4}+\frac{6}{5}-2<0,$$
giving a contradiction. Therefore, $D$ has length exactly four, i.e., $k=4$.

Recall that if $e_i\in E(D)$ is an exterior edge for some $i\in [4]$ then either $e_{i-1}$ is an exterior edge or $e_{i+1}$ is an exterior edge. This implies that if some edge in $D$ is an exterior edge, then $D$ has at least two consecutive edges that are exterior edges. Since $G$ is not a cycle, $D$ has at most three edges that are exterior edges. Therefore, there are two or three consecutive edges of $D$ contained in $E(C)$ or no edge of $D$ is contained in $E(C)$. Suppose no edge of $D$ is an exterior edge. Then all edges of $D$ are interior edges and all vertices of $D$ have degree at least four. Recall that for each $i\in [4]$, $e_i=v_iv_{i+1}$ is contained in exactly two facial cycles, $D$ and the triangle $D_i=v_i w_i v_{i+1} v_i$. Let $S=\{v_1, v_2, v_3, v_4, w_1, w_2, w_3, w_4\}$. Then $G[S] \cong G_8$. Suppose $G\ne G_8$. This implies $V(G)\backslash S \ne \emptyset$. Without loss of generality, we may assume $v_1$ has a neighbor $u_1$ such that $u_1\notin S$. As $G$ is outerplanar, either $u_1w_1\notin E(G)$ or $u_1 w_4\notin E(G)$, and we may assume that $u_1w_1\notin E(G)$. Consider the edge $v_1v_2$ and the $1$-Lipschitz function $f:N[v_1]\cup N[v_2]\to \mathbb{R}$ as follows.
    $$
     f(x):=\begin{cases}
               -1 & \textrm{ if $x\in \{u_1, w_4\}$},\\
                0 & \textrm{ if $x\in N[v_1]\backslash \{u_1, w_4, w_1, v_2\}$},\\
                1 &\textrm{ if $x\in N[v_2]\backslash \{v_1, w_2\}$},\\
                2 & \textrm{ if $x=w_2$}.
           \end{cases}$$
Observe that $f (w_1)=1, \Delta f (v_1)=\frac{1}{d(v_1)}(-2+2)=0$, and $\Delta f (v_2)=-\frac{1}{d(v_2)}(-1+1)=0$. It follows from Theorem~\ref{thm:curvature_laplacian} that $\kappa(v_1, v_2)\le \Delta f (v_1)-\Delta f (v_2)=0$, a contradiction. Therefore, $G\cong G_8$.
\end{proof}

\begin{lemma}\label{lem:suppressing}
Let $G$ be a $2$-connected outerplanar graph. Let $v_0$ be a degree two vertex contained in a facial cycle $C:=v_0 v_1 v_2 v_3 v_0$ of length $4$ in $G$. Let $G'$ be the graph obtained from $G$ by suppressing $v_0$, i.e., removing $v_0$ from $G$ and adding $v_1 v_3$ as an edge $(G'=G-v_0+v_1v_3)$. For every $xy \in E(G')\backslash \{v_1 v_3\}$, $\kappa_{G'}(x,y)\geq \kappa_G(x,y)$.
\end{lemma}
\begin{proof}
Let $xy$ be an edge in $E(G')\backslash \{v_1v_3\}$.
Suppose $f':N_{G'}[x]\cup N_{G'}[y]\to \mathbb{R}$ is a $1$-Lipschitz function with $\nabla_{yx} f'=1$ that attains $\kappa_{G'}(x,y)$. If $\{x,y\}\cap \{v_1,v_3\} =\emptyset$, 
then it is not hard to see that $f'$ can also be made a $1$-Lipchitz function for $xy$ in $G$. It follows that 
$$\kappa_G(x,y) =\inf_{\substack{f \in Lip(1)\\ \nabla_{yx}f = 1}} \nabla_{xy} \Delta f \leq \nabla_{xy} \Delta f' =\kappa_{G'}(x,y).$$
Now without loss of generality suppose that $x=v_1$, and $f':V(G')\to \mathbb{R}$ is a $1$-Lipschitz function in $G'$ with $\nabla_{yx} f'=1$ that attains $\kappa_{G'}(x,y)$. Then we can define a $1$-Lipschitz function in $G$ with $\nabla_{yx} f=1$ as follows: let $f(v_0)=f'(v_3)$, and for every vertex $v\in V(G)\backslash \{v_0\}=V(G')$, let $f(v) = f'(v)$. Then it is not hard to see that $f$ is also $1$-Lipschitz and $\nabla_{yx}=1$. Hence similar to before, we have
$$\kappa_G(x,y) =\inf_{\substack{f \in Lip(1)\\ \nabla_{yx}f = 1}} \nabla_{xy} \Delta f \leq \nabla_{xy} \Delta f' =\kappa_{G'}(x,y).\qedhere$$
\end{proof}


\section{Proof of Theorems \ref{thm:outerplanr_2_connected_max_order} and \ref{thm:outerplanar_max_order}}\label{sec:max_order_outerplanar}

\begin{proof}[Proof of Theorem~\ref{thm:outerplanr_2_connected_max_order}]
For $n\leq 10$, we classify all $2$-connected positively LLY-curved outerplanar graphs by computer search. 
For $n\geq 11$, we show that there are no $2$-connected positively LLY-curved outerplanar graphs on $n$ vertices by induction on $n$. The base case for $n=11$ is verified by computer search.

Suppose for contradiction that $G$ is a  $2$-connected positively LLY-curved outerplanar graph on $n\geq 12$ vertices with outer cycle $C$. If $G$ is maximally outerplanar, then we are done by Theorem \ref{thm:max_outerplanar}. We may now assume that $G$ is not maximally outerplanar. We claim that $G$ is not a cycle since a cycle on at least $6$ vertices is not positively LLY-curved. Thus, $G$ contains a facial cycle $D$ of length at least $4$ that is not the outer cycle of $G$. Then by Lemma \ref{lem:4-facical-cycle_structure}, $|V(D)|=4$, i.e., $D$ is a facial cycle of length exactly $4$. Since $|V(G)|\ge 12$, $G\ne G_8$ and thus $D$ has two or three consecutive edges that are exterior edges. This implies that $D$ has some vertex of degree $2$. Let $D:=v_0 v_1 v_2 v_3 v_0$ with $v_0$ being a degree two vertex.   It follows from Lemma \ref{lem:4-facical-cycle_structure} that there are two cases for the edges of $D$:
\begin{description}
    \item Case 1: Exactly three edges in $D$ are exterior edges of $G$, which (without loss of generality) are $v_3v_0, v_0v_1$ and $v_1v_2$. This implies that $d(v_1)=2$, $d(v_2)\ge 3$, $d(v_3)\ge 3$. By Lemma \ref{lem:deg2_vtx}, $\max\{d(v_2), d(v_3)\} \leq 4$.
    \item Case 2: Exactly two edges in $D$ are exterior edges of $G$ which are $v_0v_1$ and $v_0v_3$. This implies that $d(v_1)\ge 3$ and $d(v_3)\ge 3$. By Lemma \ref{lem:deg2_vtx}, $\max\{d(v_1),d(v_3)\}\leq 4$.
\end{description}

We first claim that in Case 2, $v_1$ and $v_3$ cannot both have degree $4$. Suppose for contradiction that $d(v_1) = d(v_3)=4$. We will show that $G$ is then not positively curved. By Lemma \ref{lem:deg2_vtx}, there exist another a vertex $v_1'\in N(v_1)\backslash \{v_0\}$ with $v_1'v_2\in E(G)$, and another vertex $v_3'\in N(v_3)\backslash \{v_0\}$ with $v_3'v_2\in E(G)$. Since $d(v_1)= d(v_3) = 4 $, we have that $d(v_1')\geq 3$ and $d(v_3')\geq 3$. Note that $d(v_2)\ge 4$. We claim that the edges $v_1'v_2, v_3'v_2$ both are interior edges. Now suppose either $v_1'v_2$ or $v_3'v_2$ is an exterior edge. Without loss of generality, if  $v_1'v_2$ is an exterior edge then by Lemma \ref{lem:exterior_at_least_3}, $d(v_2)=4$ and $v_2$ has another neighbor other than $v_1'$ that is adjacent to $v_1$, which is impossible due to the outerplanarity of $G$. Then $v_1'v_2$ and $v_3'v_2$ are not exterior edges of $G$. Thus $v_2$ has degree at least $6$, and it follows from Lemma~\ref{lem:exterior_at_least_3} that $N_C(v_2)$ contains two degree $2$ vertices
(call them $v_1''$ and $v_3''$, where $v_i''$ is the degree $2$ vertex in the component containing $v_i$ in $G-\{v_2, v_0\}$ for $i=1,3$). Then consider the edge $v_3v_2$ and the $1$-Lipschitz function $f:N[v_2]\cup N[v_3] \to\mathbb{R}$ with 

  $$
     f(x):=\begin{cases}
               -1 & \textrm{ if $x \in N(v_3)\backslash \{v_0,v_2, v_3'\}$},\\
                0 & \textrm{ if $x\in \{v_0, v_3, v_3'\}$},\\
                1 &\textrm{ if $x\in N[v_2]\backslash \{v_3, v_3',v_1', v_1''\}$},\\
                2 & \textrm{ if $x\in \{v_1',v_1''\}$}.
           \end{cases}$$
It follows that
\begin{align*}
\kappa(v_3,v_2) &\leq \frac{1}{d(v_3)}\lp -1+1\rp-\frac{1}{d(v_2)}\lp -2+2\rp=0,
\end{align*}
contradicting that $G$ is positively curved. Hence we conclude that $d(v_1)$ and $d(v_3)$ cannot both have degree $4$.

Let $G'$ be the graph obtained from $G$ by suppressing $v_0$, i.e., removing $v_0$ from $G$ and adding the edge $v_1 v_3$. It suffices to show that $G'$ is positively curved, which leads to a contradiction by the induction hypothesis. Note that by Lemma \ref{lem:suppressing}, for all edges $xy\in E(G')\backslash \{v_1v_3\}$, $\kappa_{G'}(x,y)\geq \kappa_G(x,y)>0$ since $G$ is assumed to positively curved. Hence it suffices to show that $\kappa_{G'}(v_1, v_3)>0$.

For Case 1, it is straightforward to see that there is a coupling $\sigma$ between $\mu_{v_1}$ and $\mu_{v_3}$ such that $$C(\sigma)\leq (1+2)(c_{v_1}-c_{v_3}).$$
Hence by Lemma \ref{lem:coupling},
\[\kappa_{G'}(v_1, v_3)
\geq 1+ \frac{1}{d_{G'}(v_3)} -\frac{3(c_{v_1}-c_{v_3})}{\lcm(d(v_1),d(v_3))}
=\frac{4}{d_{G'}(v_3)} -\frac{1}{2}>0.
\]
For Case 2, applying Lemma \ref{lem:pos_degree_pair}, we get
$\kappa_{G'}(v_1, v_3)>0$.

Hence we conclude that $G'$ is positively curved, which contradicts the induction hypothesis. This completes the proof of Theorem \ref{thm:outerplanr_2_connected_max_order}.
\end{proof}
\begin{proof}[Proof of Theorem~\ref{thm:outerplanar_max_order}]
We may now assume that $G$ is a positively curved outerplanar graph with $\delta(G)\ge 2$ and $G$ is not $2$-connected. Let $C$ be the outer walk of $G$. Since $G$ is positively curved with $\delta(G)\ge 2$, $G$ has no cut edge and every block of $G$ is $2$-connected. Let $x\in V(G)$ be a cut vertex of $G$. Then $x$ is contained in at least two blocks of $G$ and $x$ has degree at least four as $|N_C(x)|\ge 4$. Note that for every $v\in N_C(x)$, $v$ has degree two in $G$ by Lemma~\ref{lem:exterior_at_least_3}, and the two neighbors of $v$ are adjacent by Lemma~\ref{lem:deg2_vtx}. Let $v_1, v_2\in N_C(x)$ such that $v_1, v_2$ are contained in different blocks of $G$. Consider the graph $G':=G + v_1v_2$. Observe that $G'$ is still outerplanar. We claim that $G'$ is positively curved. If suffices to show that every edge of $G'$ has positive LLY curvature.

Let $u_i$ denote the neighbor of $v_i$ other than $x$ for $i\in [2]$. Then $u_ix\in E(G)$. For any edge $uv\in E(G')$ that is not incident with $v_1, v_2$, let $f': V(G')\to \mathbb{R}$ be a $1$-Lipschitz function with $\nabla_{vu} f'=1$ that attains $\kappa_{G'}(u,v)$. Observe that $f'$ is also a $1$-Lipchitz function for $uv$ in $G$. It follows that $$0<\kappa_G(u,v)=\inf_{\substack{f \in Lip(1)\\ \nabla_{vu}f = 1}} \nabla_{uv} \Delta f \leq \nabla_{uv} \Delta f' =\kappa_{G'}(u,v).$$ Next we show that for any edge $uv\in \{v_1v_2, v_1x, v_2x, v_1u_1, v_2u_2\}$, $\kappa_{G'}(u,v)>0$. Observe that $d_{G'}(v_1)=d_{G'}(v_2)=3$. Since $v_1, v_2$ have a common neighbor $x$, $\kappa_{G'}(v_1, v_2)>0$ by Lemma~\ref{lem:pos_degree_pair}. For $v_ix\in E(G')$ ($i\in [2]$), note that $v_i$ and $x$ have exactly two common neighbors (i.e., $N_{G'}(v_i)\cap N_{G'}(x)=\{u_i, v_{3-i}\}$) and $x$ has a neighbor $u_{3-i}\notin N_{G'}[v_i]$ adjacent to $v_{3-i}$, which is a common neighbor of $v_i$ and $x$. By Theorem~\ref{thm:maxdeg}, $\Delta(G)\le 9$. Then it follows from Lemma~\ref{lem:pos_degree_pair} that $\kappa_{G'}(v_i,x)>0$ as $d_{G'}(x)=d_G(x)\le \Delta(G)\le 9$. For $v_iu_i\in E(G')$ ($i\in [2]$), $v_i$ and $u_i$ have a common neighbor $x$. 
If $d_{G'}(u_i)=d_G(u_i)\le 3$, then by Lemma \ref{lem:pos_degree_pair} we have $\kappa_{G'}(u_i, v_i)>0$. 
We claim that $d_G(u_i)\le 4$, and if $d_G(u_i)=4$ then $(N_G(u_i)\cap N_G(x))\backslash \{v_i\} \ne \emptyset$ (i.e., $u_i$ has a neighbor $w\notin N[v_i]$ such that $w$ is adjacent to $x$, the common neighbor of $u_i$ and $v_i$). First we show that if $d_G(u_i)\ge 4$ then $(N_G(u_i)\cap N_G(x))\backslash \{v_i\}\ne \emptyset$. Suppose $d_G(u_i)\ge 4$ and $(N_G(u_i)\cap N_G(x))\backslash \{v_i\}=\emptyset$. Choose $u_i'\in N_G(u_i)\backslash \{v_i, x\}$ and $x'\in N_G(x)\backslash \{v_i, u_i\}$ such that $d_G(u_i', x')$ is minimum. Since $(N_G(u_i)\cap N_G(x))\backslash \{v_i\}=\emptyset$, $d_G(u_i', x')\ge 1$. Consider the edge $u_ix\in E(G)$ and the $1$-Lipschitz function $f: N_G[u_i] \cup N_G[x] \to \mathbb{R}$ with
 $$
     f(v):=\begin{cases}
               -1 & \textrm{ if $v\in N_G(u_i)\backslash \{x, v_i, u_i'\}$},\\
                0 & \textrm{ if $v\in \{u_i,  v_i, u_i'\}$},\\
                1 &\textrm{ if $v\in N_G[x]\backslash\{u_1, v_1, v_2, u_2\}$},\\
                2 & \textrm{ if $v\in \{v_{3-i},u_{3-i}\}$}.
           \end{cases}$$
 It follows that 
\begin{align*}
\kappa_G(u_i,x) & \le  \Delta f(u_i)-\Delta f(x)\\
&=\frac{1}{d_G(u_i)}\lp -(d_G(u_i)-3)+1 \rp -\frac{1}{d_G(x)}\lp -2+2\rp\\
&= \frac{4}{d_G(u_i)}-1 \le \frac{4}{4}-1=0, 
\end{align*}
a contradiction. Thus if $d_G(u_i)\ge 4$, then $(N_G(u_i)\cap N_G(x))\backslash \{v_i\}\ne \emptyset$. We show that $d_G(u_i)\le 4$. If $d_G(u_i)\ge 5$, let $w\in N_G(u_i)\cap N_G(x))\backslash \{v_i\}$. Observe that $d_G(x)\ge 5$ as $\{w, u_1, v_1, v_2, u_2\}\subseteq N_G(x)$. Consider the edge $u_ix\in E(G)$ and the $1$-Lipschitz function $f: N_G[u_i] \cup N_G[x] \to \mathbb{R}$ with
 $$
     f(v):=\begin{cases}
               -1 & \textrm{ if $v\in N_G(u_i)\backslash \{x,v_i, w\}$},\\
                0 & \textrm{ if $v\in \{u_i, v_i, w\}$},\\
                1 &\textrm{ if $v\in N_G[x]\backslash\{u_1, v_1, v_2, u_2, w\}$},\\
                2 & \textrm{ if $v\in \{v_{3-i},u_{3-i}\}$}.
           \end{cases}$$
Note that
\begin{align*}
\kappa_G(u_i,x) & \le  \Delta f(u_i)-\Delta f(x)\\
&=\frac{1}{d_G(u_i)}\lp -(d_G(u_i)-3)+1 \rp -\frac{1}{d_G(x)}\lp -3+2\rp\\
&= \frac{4}{d_G(u_i)}+\frac{1}{d_G(x)}-1 \le \frac{4}{5}+\frac{1}{5}-1=0, 
\end{align*}
contradicting that $\kappa_G(u_i,x)>0$. Therefore, we obtain that $d(u_i)\leq 4$ and and if $d_G(u_i)=4$ then $(N_G(u_i)\cap N_G(x))\backslash \{v_i\} \ne \emptyset$.
It then follows from Lemma~\ref{lem:pos_degree_pair} that $\kappa_{G'}(v_i, u_i)>0$. Therefore, $G'=G+v_1v_2$ is a positively curved outerplanar graph with $\delta(G')\ge 2$. If $G'$ is $2$-connected, we know $|V(G)|=|V(G')|\le 10$ by Theorem~\ref{thm:outerplanr_2_connected_max_order}. Suppose $G'$ is not $2$-connected. We repeat the similar process to $G'$ by adding an edge between two specified degree $2$ vertices, and the result graph is still a positively curved outerplanar graph with minimum degree at least two. We repeat adding such edges until we get a $2$-connected positively curved outerplanar graph $H$. Note that $|V(H)|=
|V(G)|$. Therefore, $G$ has at most $10$ vertices as $|V(H)|\le 10$. 
\end{proof}

\section{Appendix}

By utilizing Sagemath and Nauty/geng, we can identify all positively curved outerplanar graphs with a minimum degree of at least 2. Given our proof that the order of these graphs does not exceed 10, we have determined there are 59 such graphs in total. Below, we provide a list of these graphs, with edges labeled according to the LLY Ricci curvatures.

\begin{figure}
    \centering
    \includegraphics[width=1.0\linewidth]{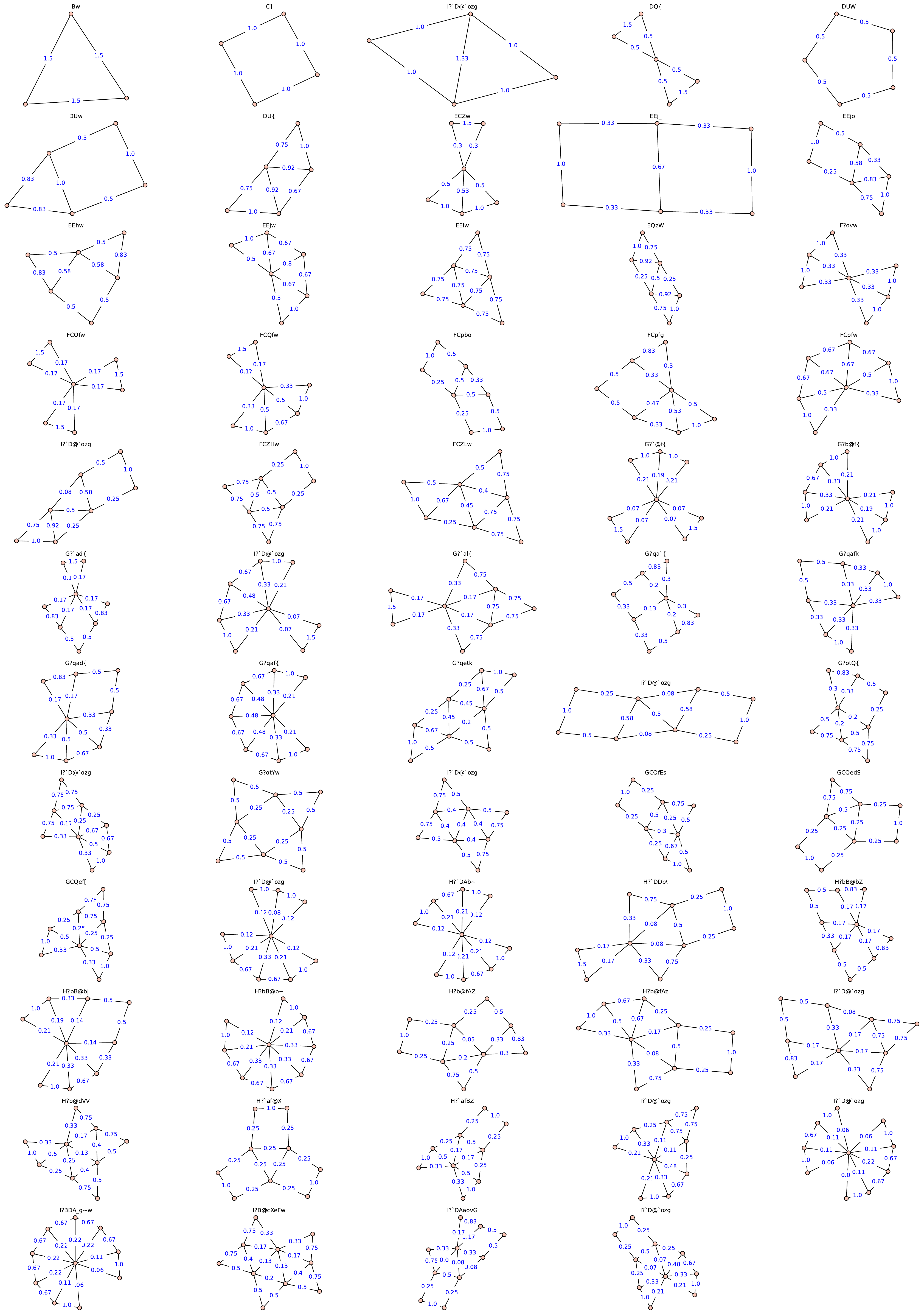}
    \caption{The complete list of positively curved outerplanar graphs with minimum degree 2.}
    \label{fig:enter-label}
\end{figure}

\end{document}

%% file: Lemma2_i.tikz
\begin{tikzpicture}[scale=1, Wvertex/.style={circle, draw=black, fill=white, scale=3}, bvertex/.style={circle, draw=black, fill=black, scale=0.3},rvertex/.style={circle, draw=red, fill=red, scale=0.2}]

\node [bvertex, label={[font=\small] below:$x$}] (x) at (-1,0) {};
\node [bvertex, label={[font=\small] below:$y$}, label={[font=\small] above:$1$}] (y) at (1,0) {};
\node [bvertex, label={[font=\small] above:$u_1$}, label={[font=\small] below:$1$}] (u1) at (0,1.73) {};
\node [bvertex, label={[font=\small] above:$u_2$}, label={[font=\small] below:$3$}] (u2) at (2,1.73) {};
\node [bvertex, label={[font=\small] below:$u_3$}, label={[font=\small] above:$3$}] (u3) at (3,0) {};
\node [bvertex, label={[font=\small] above:$x'$}, label={[font=\small] below:$4$}] (x') at (-2,1.73) {};

\draw (x) -- (y);
\draw (x) -- (u1);
\draw (y) -- (u1);
\draw (u2) -- (u1);
\draw (y) -- (u2);
\draw (y) -- (u3);
\draw (x) -- (x');

\end{tikzpicture}	

%% file: Lemma2_ii.tikz
\begin{tikzpicture}[scale=1, Wvertex/.style={circle, draw=black, fill=white, scale=3}, bvertex/.style={circle, draw=black, fill=black, scale=0.3},rvertex/.style={circle, draw=red, fill=red, scale=0.2}]

\node [bvertex, label={[font=\small] below:$x$}] (x) at (-1,0) {};
\node [bvertex, label={[font=\small] below:$y$}, label={[font=\small] above:$c_x-c_y$}] (y) at (1,0) {};
\node [bvertex, label={[font=\small] above:$u_1$}, label={[font=\small] below:$c_x-c_y$}] (u1) at (0,1.73) {};
\node [bvertex, label={[font=\small] above:$u_2$}, label={[font=\small] below:$c_y$}] (u2) at (2,1.73) {};
\node [bvertex, label={[font=\small] below:$u_3$}, label={[font=\small] above:$c_y$}] (u3) at (3,0) {};
\node [label={[font=\small] above:$......$}] (uk'') at (2,-1.73) {};

\node [bvertex, label={[font=\small] below:$u_{k-1}$}, label={[font=\small] above:$c_x-c_y$}] (uk') at (0,-1.73) {};

\draw (x) -- (y);
\draw (x) -- (u1);
\draw (y) -- (u1);
\draw (u2) -- (u1);
\draw (y) -- (u2);
\draw (y) -- (u3);
\draw (x) -- (uk');
\draw (y) -- (uk');

\end{tikzpicture}	

%% file: G8.tikz
\begin{tikzpicture}[scale=1, Wvertex/.style={circle, draw=black, fill=white, scale=3}, bvertex/.style={circle, draw=black, fill=black, scale=0.3},rvertex/.style={circle, draw=red, fill=red, scale=0.2}]

\node [bvertex, label={[font=\small] above:$v_1$}] (v1) at (-1,1) {};
\node [bvertex, label={[font=\small] above:$v_2$}] (v2) at (1,1) {};
\node [bvertex, label={[font=\small] below:$v_3$}] (v3) at (1,-1) {};
\node [bvertex, label={[font=\small] below:$v_4$}] (v4) at (-1,-1) {};
\node [bvertex, label={[font=\small] above:$w_1$}] (w1) at (0,2) {};
\node [bvertex, label={[font=\small] right:$w_2$}] (w2) at (2,0) {};
\node [bvertex, label={[font=\small] below:$w_3$}] (w3) at (0,-2) {};
\node [bvertex, label={[font=\small] left:$w_4$}] (w4) at (-2,0) {};

\draw (v1) -- (v2);
\draw (v2) -- (v3);
\draw (v3) -- (v4);
\draw (v4) -- (v1);
\draw (v1) -- (w1);
\draw (w1) -- (v2);
\draw (v2) -- (w2);
\draw (w2) -- (v3);
\draw (v3) -- (w3);
\draw (w3) -- (v4);
\draw (v4) -- (w4);
\draw (w4) -- (v1);

\end{tikzpicture}	